\documentclass[final,hidelinks,onefignum,onetabnum]{siamart251216}

\usepackage{lipsum}
\usepackage{amsfonts}
\usepackage{graphicx}
\usepackage{epstopdf}
\usepackage{algorithmic}

\ifpdf
  \DeclareGraphicsExtensions{.eps,.pdf,.png,.jpg}
\else
  \DeclareGraphicsExtensions{.eps}
\fi

\newsiamremark{remark}{Remark}
\newsiamremark{hypothesis}{Hypothesis}
\crefname{hypothesis}{Hypothesis}{Hypotheses}
\newsiamremark{claim}{Claim}
\newsiamremark{fact}{Fact}
\crefname{fact}{Fact}{Facts}

\headers{Inexact Uzawa-Double Deep Ritz Method }{E. Benny-Chacko, I. Brevis, L. Espath and K. G. van der Zee }

\title{Inexact Uzawa-Double Deep Ritz Method for Weak Adversarial Neural Networks}

\author{Emin Benny-Chacko\thanks{School of Mathematical Sciences, University of Nottingham, United Kingdom 
(\email{Emin.BennyChacko1@nottingham.ac.uk}, \email{Ignacio.Brevis1@nottingham.ac.uk}, \email{Luis.Espath@nottingham.ac.uk}, \email{KG.vanderZee@nottingham.ac.uk}).}
\and Ignacio Brevis \footnotemark[1]
\and Luis Espath\footnotemark[1]
\and  Kristoffer G. van der Zee\footnotemark[1] }

\usepackage{amsopn}

\usepackage{tikz}
\usepackage{subcaption}
\usetikzlibrary{arrows.meta}
\newcommand{\circled}[1]{%
  \tikz[baseline=(char.base)]{
    \node[shape=circle, draw=black, text=black, inner sep=1pt, line width=0.4pt] (char) {#1};
  }%
}
\ifpdf
\hypersetup{
  pdftitle={AN EXAMPLE ARTICLE*},
  pdfauthor={D. Doe, P. T. Frank, and J. E. Smith}
}
\fi

\begin{document}

\maketitle

\begin{abstract}
Residual minimization in dual norms is central to Weak Adversarial Neural Network (WAN) approaches for solving partial differential equations (PDEs). This framework naturally leads to saddle-point problems whose numerical solutions can be highly unstable depending on the underlying iterative scheme. Motivated by this structure, we propose and analyze the Uzawa Double Deep Ritz Method, a deep PDE solver that integrates neural network approximations with the classical Uzawa iteration. The proposed method is built around two coupled update rules performed at each iteration: a residual update, obtained by minimizing a Ritz functional associated with the dual problem, and a solution update, obtained by minimizing a Ritz functional driven by the current residual. Both variables are represented by neural networks, mirroring the classical Uzawa architecture for saddle-point problems. By replacing the adversarial min-max optimization of WAN with a sequence of Deep Ritz minimization problems, our study theoretically proves that the proposed method acts as an iterative scheme for solving the WAN formulation. Furthermore, we establish a comprehensive convergence theory for an inexact Uzawa scheme where both subproblems are solved approximately. This analysis extends to practical gradient-based implementations, providing rigorous stability and convergence guarantees for both single and multiple-gradient step update strategies. Numerical experiments validate our theoretical findings and demonstrate the robustness of the proposed approach.
\end{abstract}


\section{Introduction}
Partial differential equations (PDEs) provide a fundamental mathematical framework for modeling a wide range of phenomena in science and engineering, including diffusion, elasticity, fluid flow, and electromagnetics. Since analytical solutions are rarely available for realistic applications, the construction of accurate, stable, and efficient numerical approximation methods remains a central objective of scientific computing. Among the many available approaches, minimal residual (MinRes) formulations occupy a distinguished position as they measure approximation quality directly through the residual of the governing equation \cite{doi:10.1137/1.9781611977738,johnson2009numerical,larson2013finite,benzi2005numerical}. Consider the variational problem: Find \(u \in U\) such that
\begin{equation*}
b(u,v)=\ell(v), \qquad \forall\, v\in V,
\end{equation*}
where \(U\) and \(V\) are real Hilbert trial and test spaces, respectively, \(b:U\times V\to\mathbb{R}\) is a bilinear form, and \(\ell:V\to\mathbb{R}\) is a continuous linear functional. The associated MinRes problem seeks
\begin{equation*}
u^*=\arg\min_{u\in U}\|\ell-Bu\|_{V^*}
=\arg\min_{u\in U}\sup_{v\in V}
\frac{\ell(v)-b(u,v)}{\|v\|_V},
\end{equation*}
where \(B:U\to V^*\) is the operator induced by \(b\) and $V^*$ denotes the dual of $V$. Thus, the objective is to identify the trial function \(u\) whose residual \(\|\ell-Bu\|_{V^*}\) is smallest in the dual norm of the test space. This principle underlies several robust numerical methods and is closely connected to quasi-optimality \cite{doi:10.1137/1.9781611977738,brevis2022neural}, reliable error control \cite{brevis2021machine}, and mixed variational formulations \cite{doi:10.1137/1.9781611977738}. Furthermore, when variational frameworks are solved via residual minimization with neural network, this framework has been extended to establish rigorous error estimates for numerical quasi-minimizers \cite{shin2023error}.

The emergence of deep learning~\cite{aggarwal2018neural,higham2019deep,bengio2017deep} has stimulated a new generation of PDE solvers in which classical approximation spaces are replaced by trainable neural networks \cite{huang2025partial}. Within this setting, residual minimization remains especially attractive because it preserves the variational structure of the original problem while enabling mesh-free approximations. Early neural network approaches have already explored residual-based training for differential equations \cite{dissanayake1994neural,lagaris1998artificial}. More recently, Physics-Informed Neural Networks (PINNs) enforced strong-form residuals directly in the loss function \cite{raissi2019physics,sirignano2018dgm}, while variational methods such as VPINNs, Robust VPINNs, and Deep Ritz formulations employed weak residuals or energy principles \cite{kharazmi2021hp,rojas2024robust,yu2018deep}. These developments highlight the growing importance of combining expressive neural approximations with mathematically meaningful optimization objectives.

A particularly relevant framework from the viewpoint of dual-norm residual minimization is the Weak Adversarial Network (WAN) method \cite{zang2020weak}. In WAN, one neural network approximates the trial solution, while a second network acts as an adaptive test function that seeks the most critical residual direction. This leads naturally to the saddle-point problem
\begin{equation*}
\min_{u_\theta\in U_\theta}\sup_{v_\eta\in V_\eta}
\frac{\langle \ell-Bu_\theta,v_\eta\rangle_{V^*,V}}{\|v_\eta\|_V}.
\end{equation*}
 where \( u_{\theta} \in U_{\theta} \subset U \) and \( v_{\eta} \in V_{\eta} \subset V \) denote neural 
networks, while \(U_{\theta}\) and \(V_{\eta}\) 
are the corresponding neural network sets, and \(\theta\) and \(\eta\) are the trainable 
parameters.

Although conceptually appealing, WANs involve numerical instability during the min-max optimization. Specifically, as the resdiual approaches zero, identifying the supremizer becomes ill-posed, which was mitigated in \cite{bertoluzza2024wan} and \cite{uriarte2023deep}. \cite{bertoluzza2024wan} provides a theoretical analysis of solution stability and approximation bounds while proposing stabilized WAN variants that avoid unstable direct normalization and improve boundary condition enforcement. The Double Deep Ritz Method introduced in \cite{uriarte2023deep}, replaces the adversarial structure with two coupled Deep Ritz energy minimization for the primal and test networks; however, this method requires nested inner minimization loops, making each iteration computationally demanding and can be sensitive to the accuracy of inner solves. This motivates the development of more robust solution strategies grounded in classical numerical analysis.

The motivation for the present work is rooted in the structural equivalence among several formulations of residual minimization, as schematically summarized in the diagram of relationships in \cref{fig:equivalence}. In the ideal continuous setting (the top row of \cref{fig:equivalence}), the saddle-point problem associated with MinRes is equivalent to both a mixed variational formulation and a constrained Ritz minimization problem through their common optimality conditions. Consequently, classical iterative solvers such as the Uzawa method \cite{uzawa1958iterative,bacuta2006unified,bramble1997analysis,badea2025convergence} may be viewed not only as saddle-point solvers, but also as iterative procedures that converge to the exact MinRes solution as the iteration count $k \to \infty$, as illustrated by relation \circled{1} in \cref{fig:equivalence}.

When moving from continuous spaces to neural network sets (the transition from the top row to the bottom row via the approximation limits $V_\eta \to V$ in relations \circled{2} and \circled{4}), this underlying variational structure is preserved. This yields two methods: the Weak Adversarial Network (WAN) objective (bottom left) and the Uzawa Double Deep Ritz method (bottom right). Crucially, establishing relation \circled{3} between a fully discretized neural network setting and the WAN formulation is fundamental, as the Uzawa Double Deep Ritz method solves the WAN problem. Hence, the Uzawa Double Deep Ritz method serves as an iterative realization of the WAN formulation. Related ideas have recently been extended to partially discretized neural settings, where equivalence results were established for residual-Uzawa formulations \cite{alsobhi2025neural}. By mapping out these connections, \cref{fig:equivalence} highlights how our proposed framework replaces adversarial min-max training with a robust, sequential residual-correction and solution-update procedure grounded in classical saddle-point theory.

\begin{figure}[htbp]
    \centering
\scalebox{0.8}{    
\begin{tikzpicture}[
>=Latex,
box/.style={draw, rounded corners=4pt, very thick, align=center, minimum width=5cm, minimum height=2.8cm, inner sep=8pt},
arr/.style={->, thick}
]

\node[box] (wan) at (0,0) {\textbf{WAN}\\[2pt]
$u_\theta = \arg\min\limits_{u_\theta\in U_\theta}\sup\limits_{v_\eta\in V_\eta}
\dfrac{\langle \ell-Bu_\theta,v_\eta\rangle_{V^*,V}}{\|v_\eta\|_V}$};

\node[box] (minres) at (0,4) {\textbf{MinRes}\\[2pt]
$u_h = \arg\min\limits_{u_h\in U_h}\sup\limits_{v\in V}
\dfrac{\langle \ell-Bu_h,v\rangle_{V^*,V}}{\|v\|_V}$};

\node[box] (deep) at (8.6,0) {\textbf{Uzawa Double Deep Ritz}\\[2pt]
$\dfrac12\|r_\eta^k\|^2_V-\ell(r_\eta^k)+b(u_\theta^k,r_\eta^k)\to\min\limits_{r^k\in V_\eta}$\\[2pt]
$\dfrac12\|u_\theta^{k+1}\|^2_U-(u_\theta^{k+1},u_\theta^k)_U-\tau b(u_\theta^{k+1},r_\eta^k)\to\min\limits_{u_\theta^{k+1} \in U_\theta}$};

\node[box] (uzawa) at (8.6,4) {\textbf{Uzawa}\\[2pt]
$\dfrac12\|r^k\|^2_V-\ell(r^k)+b(u_h^k,r^k)\to\min\limits_{r^k\in V}$\\[2pt]
$(u_h^{k+1}-u_h^k,w_h)_U-\tau b(w_h,r^k)=0~\forall w_h \in U_h$};

\draw[arr] (wan) -- node[left,pos=.5] {$V_\eta\to V$} (minres);
\draw[arr] (wan) -- node[right,pos=.5] {\circled{4}} (minres);
\draw[arr] (deep) -- node[left,pos=.5] {$V_\eta\to V$} (uzawa);
\draw[arr] (deep) -- node[right,pos=.5] {\circled{2}} (uzawa);
\draw[arr] (deep) -- node[above,pos=.5] {$k\to\infty$ } (wan);
\draw[arr] (deep) -- node[below,pos=.5] {\circled{3}} (wan);
\draw[arr] (uzawa.west) to[out=180,in=0] node[above,pos=.5] {$k\to\infty$} (minres.east);
\draw[arr] (uzawa.west) to[out=180,in=0] node[below,pos=.5] {\circled{1}} (minres.east);

\end{tikzpicture}
}
    
    \caption{Framework illustrating the relationships among WAN, MinRes, Uzawa, and Uzawa Double Deep Ritz approaches.}
    \label{fig:equivalence}
\end{figure}

In this work, we propose and analyze the \emph{Uzawa Double Deep Ritz Method}, a fully neural network variational solver that combines Deep Ritz approximations with the classical Uzawa iteration. Through the theoretical analysis, we bridge the relation \circled{3}, formally proving that the proposed method is an iterative scheme for solving WAN objective in the neural network setting. At each outer iteration, the method performs two coupled updates: a residual update obtained by minimizing a Ritz functional associated with the dual problem, followed by a solution update obtained by minimizing a Ritz functional driven by the current residual. Both the solution and residual variables are represented by neural networks. The resulting framework preserves the mesh-free flexibility of neural approximation while replacing unstable adversarial dynamics with sequential energy-based optimization steps. It may be viewed as a fully neural counterpart of the recent residual-Uzawa approach \cite{alsobhi2025neural} and is closely related in spirit to Ritz-Uzawa Neural Network formulations \cite{herrera2026runns,makridakis2024deep}.

The main contributions of this paper are as follows:
\begin{itemize}
    \item We establish that the Uzawa Double Deep Ritz method can be interpreted as an iterative realization of the WAN formulation (\cref{thm:uzawa_wan}), thereby connecting adversarial residual minimization with classical saddle-point iterations (relation \circled{3} in \cref{fig:equivalence}).

    \item We prove convergence of an inexact Uzawa scheme in which both update steps are performed only approximately (\cref{thm:inexact}), provided the approximation errors at each step are suitably controlled.

    \item We establish convergence results when the inexact updates are computed by gradient-based optimization (\cref{thm:uddr} and \cref{thm:nmconvg}). In particular, we analyze both a single-gradient-step regime, which is structurally equivalent to the Arrow-Hurwicz method \cite{benzi2005numerical}, and a multi-step regime in which several gradient iterations are performed within each update block.
\end{itemize}
In the remainder of this paper, \Cref{sec:methodology} presents the methodology of the proposed framework. It begins with the abstract variational setting and the dual-residual minimization principle, then derives the equivalent saddle-point formulation and recalls the classical Uzawa iteration together with its convergence properties. The section also explains how the weak adversarial neural network formulation arises from restricting the trial and test spaces to neural network sets. \Cref{sec:uzawa_drr} introduces the Uzawa Double Deep Ritz method, where both the residual and solution variables are represented by neural networks, and establishes its interpretation as an iterative realization of the weak adversarial network framework. \Cref{sec:Theoretical_Results} develops the theoretical analysis of the method by studying an inexact Uzawa scheme in which both updates are computed approximately. Convergence results are proved for controlled inexactness, for single gradient-step updates, and for multiple inner gradient iterations. \Cref{sec:Training-alg} describes the neural network approximation and practical training algorithms, including both block-gradient and full-gradient strategies for solving the two Deep Ritz minimization problems. Finally, \Cref{sec:num_exp} reports numerical experiments for representative one-dimensional and two-dimensional PDE problems, demonstrating the convergence and effectiveness of the proposed solver.
\section{Methodology}
\label{sec:methodology}
In this section, we first describe the abstract problem setting and recall the dual-residual minimization framework that motivates our approach. We then express the problem in its equivalent saddle-point form and review the Uzawa iteration at the continuous level, together with its standard convergence property, following the ideas in \cite{makridakis2024deep}. The resulting framework provides the foundation for the neural network Uzawa scheme introduced in \cref{sec:uzawa_drr}.

\subsection{Abstract Setting}
We consider the following abstract variational formulation:
\begin{equation}\label{variational_formulation}
    \text{Find } u \in U \text{ such that:   } b(u, v) = \ell(v), \quad \forall\, v \in V.
\end{equation}
where \( U \) and \( V \) are real Hilbert trial and test spaces respectively, $b:U \times V \rightarrow \mathbb{R}$ is a bilinear form, and $\ell:V \rightarrow\mathbb{R}$ is a continuous linear functional. Equivalently in operator form:
\begin{equation}\label{operator_form}
    \text{Find } u \in U \text{ such that: } Bu = \ell,
\end{equation}
where $B: U \rightarrow V^*$ is the operator defined by \[  \langle Bu, v \rangle_{V^*,V} = b(u, v) \quad \text{for all } u \in U,\, v \in V.
\]and $\ell \in V^*$, where $V^*$ denotes the dual of $V$.
Also, $B$ is continuous and bounded, i.e.,
\begin{equation}\label{bound_B}
    m\|u\|_U \leq \|Bu\|_{V^*}\leq M\|u\|_U, \quad u\in U
\end{equation}
for some positive constants $m\leq M$.
\subsection{Residual Minimization and Saddle point problem}
The equivalent Minimal Residual formulation for $\cref{variational_formulation}$ is,
\begin{equation}\label{Min_Res}
u^*= arg\min_{u\in U}\|\ell-Bu\|_{V^*}=\arg\min_{u\in U}\max_{v\in V}\frac{\ell(v)-b(u,v)}{\|v\|_V}.  
\end{equation}
The associated least-squares functional is defined as:
\begin{equation}
    J(u)=\frac{1}{2}\|\ell-Bu\|_{V^*}^2 \quad \forall\, u\in U.
\end{equation}
The minimal residual approximation is then obtained by solving:
\begin{equation}
    u^*= \arg\min_{u\in U}J(u).
\end{equation}
That is, we need to find the $u\in U$ that minimizes the loss function \(J(u)\), which ensures that $u$ satisfies the variational problem $\cref{variational_formulation}$.

\noindent This minimization problem is equivalent to a mixed (saddle-point) formulation as explained in \cite{doi:10.1137/1.9781611977738}, i.e.,
\begin{equation}\label{mixed problem_cont}
 \begin{cases}
 \text{Find  }r \in V, \ u \in U, \text{ such that:} \\
(r, v)_V + b(u, v) = \ell(v), \quad v \in V, \\
b(w, r) = 0, \quad  w \in U.
\end{cases}   
\end{equation}
\noindent From $\cref{Min_Res}$ and $ \cref{mixed problem_cont}$, the minimal residual problem is explicitly seen to be a min–max (saddle point) problem. Consequently, iterative methods developed for saddle point systems, such as Uzawa algorithms, become natural candidates for solving them. This observation motivates the use of Uzawa approaches in the context of minimal residual formulations.
\subsection{Uzawa Method}
To solve the saddle-point system \cref{mixed problem_cont} iteratively, we use the Uzawa algorithm where residual $r$ and solution $u$ are updated iteratively. 

\noindent We can reformulate \cref{mixed problem_cont} for $\tau>0$ as
\begin{equation}\label{uzawa}
 \begin{cases}
\text{Find  }r^k \in V, \ u^{k+1} \in U, \text{ such that:}\\
(r^k, v)_V + b(u^k, v) = \ell(v), \quad v \in V, \\
(u^{k+1},w)_U=\tau~b(w, r^k)+(u^k,w)_U , \quad  w \in U.
\end{cases}   
\end{equation}
The Uzawa Iterative Algorithm is as follows:
\begin{algorithm}[htbp]
\caption{Uzawa Algorithm}
\label{alg:uzawa}\label{uzawa_iterate}
\begin{algorithmic}[1]
\REQUIRE initial guess $u^{0}\in U$; step size $\tau\in(0,2/M^{2})$; maximum iterations $N$
\FOR{$k = 0,1,2,\dots,N-1$}
  \STATE Solve $r^{k}\!\in\!V$ such that:
  \begin{equation}\label{r_update}
     (r^{k},v)_V=\ell(v)-b(u^{k},v) \quad \forall\, v\in V  
  \end{equation}
  \STATE Update $u^{k+1}\!\in\!U:$
  \begin{equation}\label{u_update}
      (u^{k+1},w)_U=(u^{k},w)_U+\tau~b(w,r^{k}) \quad \forall\, w\in U 
    \end{equation} 
\ENDFOR
\RETURN $u^{N},r^{N-1}$
\end{algorithmic}
\end{algorithm}

\noindent We can also reformulate the Uzawa iterations \cref{r_update}, \cref{u_update} in terms of operators as:
\begin{equation}\label{operator_iterate}
 \begin{cases}
\langle R_V r^k,v\rangle_{V^*,V} + \langle Bu^k,v \rangle_{V^*,V} = \langle \ell,v \rangle_{V^*,V}\\
(u^{k+1},w)_U=(u^{k},w)_U+\tau~(R_U ^{-1}B^* r^{k},w)_U,
\end{cases} 
\end{equation}
which can also be written as 
\begin{equation}\label{uzawa_operator_update}
\begin{cases}
R_V r^k =  \ell -  Bu^k\\
u^{k+1}=u^{k}+\tau~R_U ^{-1}B^* r^{k},
\end{cases} 
\end{equation}
where $R_V \colon V \rightarrow V^{*}$ is Riesz operator on $V$, $R_U^{-1} \colon U^{*}\rightarrow U$ is the inverse Riesz operator on $U$ and  \( B^{*} : V \to U^* \) be the dual operator of \( B \), defined by
\[
  \langle B^* v, u \rangle_{U^*,U} =  \langle v, Bu \rangle_{V,V^*} =  \langle Bu, v \rangle_{V^*,V} = b(u, v) \quad \text{for all } u \in U,\, v \in V.
\]
The classical convergence result for Uzawa \cref{alg:uzawa} followed from \cite{makridakis2024deep} is as follows:
\begin{theorem}[Classical Result: Convergence of Exact Uzawa Method] \label{thm3.1}\\
Let \(u^k \in U \text{ and } r^k \in V \text{ for }  k= 0,1,2,\cdots\) be generated through the Uzawa iterative method \cref{uzawa_iterate}. Suppose \(u^* \in U \text{ and } r^* \in V\)  be the saddle point of the saddle point problem \cref{mixed problem_cont}. Then for \( \tau \in (0, \frac{2}{M^2})\) we have that 
    \(u^k \rightarrow u^* \in U, r^k\rightarrow r^*\in V.\)
\end{theorem}
\begin{proof}
    See Appendix \ref{sec:appendix}
\end{proof}
\subsection{Residual Minimization and Weak Adversarial Neural Network formulation}
To obtain a computable scheme, we restrict the infinite-dimensional spaces $U$ and $V$ in \cref{Min_Res} to neural network sets $U_\theta \subset U$ and $V_\eta \subset V$ parametrized by neural networks. This leads to the min-max problem
\begin{equation}\label{min-max_wan}
\min_{u_\theta \in U_\theta}
\sup_{v_\eta \in V_\eta}
\left(
\frac{\langle \ell - Bu_\theta , v_\eta \rangle_{V^*,V}}{\|v_\eta\|_V}
\right).  
\end{equation}
which corresponds to the Weak Adversarial Network (WAN) formulation.\\
\noindent In particular, the adversarial maximization in $v_\eta$ 
corresponds to the residual variable $r$ in the Uzawa framework, 
while the minimization in $u_\theta$ represents the solution update $u$. This structural correspondence suggests interpreting WAN training as an inexact Uzawa iteration(detailed in \Cref{sec:Theoretical_Results}), while both the update variables are restricted to neural network sets and subproblems are solved approximately. Such an interpretation provides a way in which convergence and stability properties of neural adversarial solvers may be studied.

Motivated by this observation, we propose a neural network Uzawa framework 
in which both the trial variable $u$ and the residual variable $r$ 
are represented by two interacting neural networks.
\section{Approximation with Neural Networks: Uzawa Double Deep Ritz Method}\label{sec:uzawa_drr}

In this section, we propose the Uzawa Double Deep Ritz algorithm 
(\cref{alg:uzawa_DRR}) to construct a Deep PDE Solver. The key idea is to embed trainable neural network approximations 
into both the residual and solution update steps of Uzawa. Also, we show how the Uzawa Double Deep Ritz Method is an iterative realization of the WAN formulation.
\subsection{Uzawa Double Deep Ritz Method}
The Uzawa iterative scheme \cref{uzawa_iterate} consists of two update rules, one for the residual \(r\) and one for the solution \(u\). Each of these updates can be interpreted as a Ritz minimization problem.

We first consider the update of \(r\) given in \cref{r_update}. For a fixed \(u^{k} \in U\), this update is equivalent to solving the Ritz minimization
\begin{equation}\label{ritz_r}
    r^{k} = \arg\min_{r \in V} \, \mathcal{L}_{u^{k}}(r),
    \qquad
    \mathcal{L}_{u^{k}}(r)
    = \frac{1}{2}\|r\|_{V}^{2} - \ell(r) + b(u^{k}, r).
\end{equation}
Thus, the variable \(r^{k}\) is obtained as the unique minimizer of the quadratic functional \(\mathcal{L}_{u^{k}}\).

Next, we examine the update of \(u\) given in \cref{u_update}. For fixed \(r^{k} \in V\) and the previous iterate \(u^{k}\), this step is also characterized by a Ritz minimization problem:
\begin{equation}\label{ritz_u}
    u^{k+1} = \arg\min_{u \in U} \, \mathcal{L}_{r^{k}}(u),
    \qquad
    \mathcal{L}_{r^k}(u)=\frac{1}{2} \| u \|_U^2 - \tau b(u, r^{k}) - (u^{k}, u)_U.
\end{equation}

Hence, each iteration of the Uzawa algorithm can be viewed as performing two successive minimizations:
\begin{itemize}
    \item Minimizing the Ritz functional \cref{ritz_r} to obtain \(r^{k}\).
    \item Minimizing the Ritz functional \cref{ritz_u} to compute \(u^{k+1}\).
\end{itemize}
When minimizer \(u\) and \(r\) are approximated using neural network parameterizations, this naturally leads to a \emph{Uzawa Double Deep Ritz method}.

Let \( u_{\theta} \in U_{\theta} \subset U \) and \( r_{\eta} \in V_{\eta} \subset V \) denote neural 
network representations of the $u$ and $r$, where \(U_{\theta}\) and \(V_{\eta}\) 
are the corresponding neural network sets and \(\theta\) and \(\eta\) are the trainable 
parameters. These parametric spaces approximate the exact solution spaces while enabling efficient 
optimization through gradient-based training.
\begin{algorithm}[htbp]
\caption{Uzawa Double Deep Ritz Algorithm}
\label{alg:uzawa_DRR}
\begin{algorithmic}[1]
\REQUIRE initial guess $u^{0}_{\theta}\in U_{\theta}$; step size $\tau>0$; maximum iterations $N$
\FOR{$k = 0, 1, 2, \dots,N-1$}
    \STATE 
    \(
        r_{\eta}^{k} 
        \gets \arg\min_{r_\eta \in V_\eta}
        \left\{
            \frac{1}{2}\| r_{\eta} \|^{2}_V
            - \ell(r_{\eta})
            + b\!\left( u_{\theta}^{k}, r_{\eta} \right)
        \right\}
    \)
    \STATE
    \(
        u_{\theta}^{k+1}
        \gets \arg\min_{u_\theta \in U_\theta}
        \left\{
            \frac{1}{2}\| u_{\theta} \|^{2}_U
            - \tau b\!\left( u_{\theta}, r_{\eta}^{k} \right)
            - \big( u_{\theta}^{k}, u_{\theta} \big)_U
        \right\}
    \)
\ENDFOR
\RETURN $u_\theta^N, r_\eta^{N-1}$
\end{algorithmic}
\end{algorithm}
This iterative scheme performs two deep Ritz minimizations per step. We refer to this approach as 
the \emph{Uzawa Double Deep Ritz Method}.
\begin{theorem}\label{thm:uzawa_wan}
    The Uzawa Double Deep Ritz method is an iterative realization of the Weak Adversarial Network formulation. Any convergent sequence $\{(u_\theta^k, r_\eta^k)\}_{k \ge 0}$ generated by \cref{alg:uzawa_DRR} converges to a saddle point $(u_\theta^*, r_\eta^*)$ of the WAN min-max objective \cref{min-max_wan}.  
\end{theorem}
\begin{proof}
 Consider the update steps from Uzawa Double Deep Ritz \cref{alg:uzawa_DRR};
\begin{enumerate}
    \item Given \(u_\theta^k \in U _\theta\), \begin{equation}\label{eq:UDDR1}
      r_\eta^k= r_\eta^k(u_\theta^k)
=
\arg\min_{r_\eta \in V_\eta}
\left(
\frac12 \|r_\eta\|_V^2
-
\langle \ell - Bu_\theta^k , r_\eta \rangle _{V*,V} 
\right).  
    \end{equation}
\item With \(r_\eta^k\) fixed,
\begin{equation}\label{eq:UDDR2}
u_{\theta}^{k+1}
=
\arg\min_{u_\theta \in U_\theta}
\left\{
\frac12 \|u_\theta\|_U^2
-
\tau\, b\!\left(u_\theta, r_\eta^k\right)
-
\left(u_\theta^k, u_\theta\right)_U
\right\}.
\end{equation}
\end{enumerate}

The minimizer of the Ritz functional \cref{eq:UDDR2} is equivalently characterized as the minimizer of \(\frac{1}{2}\|u_\theta-u_\theta^k\|^2+\tau~\langle \ell - Bu_\theta,r_\eta^k \rangle _{V*,V} \), i.e. 
\begin{equation}\label{eq:UDDR3}
 u_{\theta}^{k+1}
=
\arg\min_{u_\theta \in U_\theta} \frac{1}{2}\|u_\theta-u_\theta^k\|^2_U+\tau~\langle \ell - Bu_\theta,r_\eta^k \rangle _{V*,V}
\end{equation}
Let \((u_\theta^k,r_\eta^k) \rightarrow(u_\theta^*,r_\eta^*)\) as \(k \rightarrow \infty\).
For this limiting case let us revisit each update rule, for a fixed $u_\theta$ the residual update \cref{eq:UDDR1} becomes,
\begin{equation}\label{eq:UDDR4}
  r_\eta^*= r_\eta^*(u_\theta)
=
\arg\min_{r_\eta \in V_\eta}
\left(
\frac12 \|r_\eta\|_V^2
-
\langle \ell - Bu_\theta, r_\eta \rangle _{V*,V} 
\right). 
\end{equation}
According to the equivalence of supremum and minimization formulation established in Lemma 1 of \cite{alsobhi2025neural}, we can say \cref{eq:UDDR4} is equivalent to,
\begin{equation}\label{eq:UDDR5}
    r_\eta^*(u_\theta)=\arg \sup_{v_\eta \in V_\eta} \frac{\langle \ell - Bu_\theta, v_\eta \rangle _{V*,V} }{\|v_\eta\|_V}
\end{equation}
subject to normalization,
\begin{equation}\label{eq:UDDR6}
   \|r_\eta^*(u_\theta)\|^2_V=\langle \ell - Bu_\theta, r_\eta^*(u_\theta) \rangle _{V*,V}. 
\end{equation}
Now consider the solution update \cref{eq:UDDR3} as \(k \rightarrow \infty\), it becomes,
\begin{align*}
u_\theta^*&=\arg\min_{u_{\theta} \in U_\theta}\langle \ell - Bu_\theta, r_\eta^*(u_\theta) \rangle _{V*,V} \\
&= \arg\min_{u_{\theta} \in U_\theta}\|r_\eta^*(u_\theta)\|_V^2 \quad(\text{from}~\cref{eq:UDDR6})\\
& =\arg\min_{u_{\theta}\in U_\theta}\|r_\eta^*(u_\theta)\|_V\\
& = \arg\min_{u_{\theta} \in U_\theta} \frac{\langle \ell - Bu_\theta, r_\eta^*(u_\theta)\rangle _{V*,V}}{\|r_\eta^*(u_\theta)\|_V} \quad(\text{from} ~\cref{eq:UDDR6})\\
& = \arg\min_{u_{\theta} \in U_\theta} \sup_{v_{\eta} \in V_\eta} \frac{\langle \ell - Bu_\theta, v_\eta \rangle _{V*,V} }{\|v_\eta\|_V}\quad(\text{from}~\cref{eq:UDDR5})\\
& \Leftrightarrow \text{WAN formulation}~\cref{min-max_wan}.
\end{align*}
Thus, the limit of the iterative sequence recovers the saddle point solution of the WAN formulation.
\end{proof}
\section{Theoretical Results}\label{sec:Theoretical_Results}
The Uzawa framework provides a natural and effective approach for solving the minimal residual formulation, however, in practical implementations, the exact realization of each Uzawa update is rarely
feasible. In particular, both the residual and solution updates require either the inversion of operators or the exact solution of associated variational problems, which becomes computationally prohibitive in high-dimensional or complex settings.

Motivated by these considerations, the proposed Uzawa Double Deep Ritz method is intrinsically formulated within an inexact framework, where each update is computed only approximately. This observation necessitates the development of a rigorous convergence theory tailored to the inexact Uzawa setting underlying the method.
 
In this section, we first introduce an abstract inexact Uzawa scheme and establish its convergence under controlled approximation errors. We then refine the analysis to a practically relevant regime in which each update is obtained by a number of gradient descent steps applied to the corresponding Ritz functionals. 
\subsection{Inexact Uzawa Method}
In practice, the exact Uzawa updates are typically replaced by approximate ones. 
The operator-based update \cref{uzawa_operator_update}, if computed exactly, can be expensive, particularly for large-scale or high-dimensional problems. Consequently, we replace both update steps in the Uzawa iteration with computationally feasible approximations. 
 
At iteration \(k\), the exact residual equation
\(
     r^{k} = R_V^{-1}(\ell - B u^{k})
\)
would yield the exact dual update \(r^{k}\). Since an exact solution of this problem is typically too costly, we replace it with an approximation \(r^{k}_{\delta}\) satisfying a prescribed relative accuracy (error bound $\delta$). This residual update is then used in the subsequent solution update step.

In the exact Uzawa method, the solution $u$ is updated according to
\[
    u^{k+1} = u^{k} + \tau R_{U}^{-1} B^{*} r^{k}.
\]
However, because only the approximate residual \(r^{k}_{\delta}\) is available, one first forms the corresponding
exact update with this approximate input,
\[
    u^{k+1}_{\delta} 
    = u^{k}_{\delta} + \tau R_{U}^{-1} B^{*} r^{k}_{\delta}.
\]
As with this update, evaluating this expression exactly may be computationally intensive, and we therefore
introduce an additional approximation \(u^{k+1}_{\delta,\varepsilon}\), whose deviation from 
\(u^{k+1}_{\delta}\) is controlled by a relative error bound ($\varepsilon$).

 Thus, each iteration is subsequently replaced by an inexact computation. The resulting structure forms the basis of the inexact Uzawa framework analyzed in the following section.
\begin{algorithm}[htbp]
\caption{Inexact Uzawa Iteration}
\label{alg:inexact-uzawa}
\begin{algorithmic}[1]
\REQUIRE {initial guess $u^{0}_{\delta,\varepsilon} \in U$; step size $\tau>0$}
\FOR{$k = 0,1,2,\dots$ until convergence}
    \STATE 
    {Compute $r^{k}_{\delta}$ as an approximation of  
    \( r^{k} = R_V^{-1}(\ell - B u^{k}_{\delta,\varepsilon}).\)}
    \STATE
    {Compute $u^{k+1}_{\delta,\varepsilon}$ as an approximation of  
    \(u^{k+1}_{\delta} = u^{k}_{\delta,\varepsilon}
        + \tau\, R_{U}^{-1} B^{*} r^{k}_{\delta}.\)}
\ENDFOR
\RETURN $u_{\delta,\varepsilon}, r_{\delta}$
\end{algorithmic}
\end{algorithm}
We now analyze the convergence behavior of the resulting inexact Uzawa scheme in \cref{alg:inexact-uzawa}, in which both the \(r\)-update and the \(u\)-update are performed with controlled inexactness. This provides a theoretical foundation for the approximate updates in place of the exact Uzawa updates. 

\subsection{Convergence analysis}
The convergence of the inexact Uzawa method with an approximate update on the first update only is analyzed in \cite{bacuta2006unified}.
Building on this perspective, we establish a convergence result for an inexact Uzawa scheme in which both update steps are performed approximately.
The resulting theorem shows that the overall iteration remains convergent, provided the approximation errors at each step are suitably controlled.
\begin{theorem}[Convergence of Inexact Uzawa Method]\label{thm:inexact}
\\
Assume \( \|r_{\delta}^{k}-r^{k}\| \leq \delta\|r^{k}\| \)  and $\| u^{k+1}_{\delta,\varepsilon} - u^{k+1}_{\delta} \| \leq \varepsilon \| u^{k}_{\delta,\varepsilon} - u^{k+1}_{\delta} \|$ with $\delta>0,\varepsilon>0$ and
\[\delta + \varepsilon(1+\delta)< \frac{1 - \gamma}{\tau M^2} \]
and \( \gamma =\|I-\tau R_U^{-1}B^* R_V^{-1}B \|=\max \{|1-\tau m^2|,|1-\tau M^2|\}<1\). Then,\begin{equation*}
    u^k_{\delta,\varepsilon} \rightarrow u^* \in U, r_{\delta}^k\rightarrow r^*\in V.
\end{equation*}  
\end{theorem}
\begin{proof}
Let $e^{k}_{r} = r^{k}_{\delta} - r^{*}$ and $e^{k}_{u} = u^{k}_{\delta,\varepsilon} - u^{*}$. We begin with the residual update rule. From the definition,  
\(
R_V r^{k} = \ell - Bu^{k}_{\delta,\varepsilon}
\)  
and subtracting the corresponding relation for the exact solution gives,  \[R_V (r^{k} - r^{*}) = -B(u^{k}_{\delta,\varepsilon} - u^{*}) \]  
\begin{equation} \label{eq:I}
\implies \| r^{k} - r^{*} \| \leq M \| e^{k}_u \|. 
\end{equation}  
Next, decompose the residual error as  
\(
R_V \big( r^{k} - r^{k}_{\delta} + r^{k}_{\delta} - r^{*} \big) = - B (e^{k}_{u}), 
\)
which yields  
\(
R_V (r^{k}_\delta - r^{*}) = R_V ( r^{k}_{\delta}-r^{k} ) - B(e^{k}_u).
\)  
Consequently,  
\begin{equation}
\| r^{k}_{\delta} - r^{*} \| \leq \| r^{k}_{\delta} - r^{k} \| + M \| e^{k}_u \| \leq M (1 + \delta) \| e^{k}_u \|.  
\label{eq:II}    
\end{equation} 
We now turn to the contraction property for $\| e^{k}_u \|$. The update for $u^{k+1}_{\delta}$ satisfies  
\(
u^{k+1}_{\delta} - u^{*} = u^{k}_{\delta,\varepsilon} - u^{*} + \tau R_{U}^{-1} B^{*} r^{k}_{\delta}.
\)  
By adding and subtracting $u^{k+1}_{\delta,\varepsilon}$, we obtain
\begin{align*}
u^{k+1}_{\delta} - u^{k+1}_{\delta,\varepsilon} + u^{k+1}_{\delta,\varepsilon} - u^{*} 
&= u^{k}_{\delta,\varepsilon} - u^{*} + \tau R_{U}^{-1} B^{*}(r^{k}_{\delta} - r^{*}),
\end{align*}
which leads to
\begin{align*}
u^{k+1}_{\delta,\varepsilon} - u^{*} 
&= (u^{k}_{\delta,\varepsilon} - u^{*}) + \tau R_{U}^{-1} B^{*} (r^{k}_{\delta} - r^{k}) + \tau R_{U}^{-1} B^{*} (r^{k} - r^{*}) + (u^{k+1}_{\delta,\varepsilon} - u^{k+1}_{\delta}) \\
&= (u^{k}_{\delta,\varepsilon} - u^{*}) - \tau R_{U}^{-1} B^{*} R_{V}^{-1} B (u^{k}_{\delta,\varepsilon} - u^{*}) + \tau R_{U}^{-1} B^{*} (r^{k}_{\delta} - r^{k})\\ 
 & \quad~+ (u^{k+1}_{\delta,\varepsilon} - u^{k+1}_{\delta}).
\end{align*}  

Taking norms and applying the bounds leads to  \begin{align*}
\| e^{k+1}_u \| 
&\leq \| I - \tau R_{U}^{-1} B^{*} R_{V}^{-1} B \| \, \| e^{k}_u \| + \tau M \| r^{k}_{\delta} - r^{k} \| + \| u^{k+1}_{\delta,\varepsilon} - u^{k+1}_{\delta} \| \\
&\leq \gamma \| e^{k}_u \| + \tau M \delta \| r^{k} - r^{*} \| + \varepsilon \| u^{k}_{\delta,\varepsilon} - u^{k+1}_{\delta} \| \\
&\leq \gamma \| e^{k}_u \| + \tau M \delta \| r^{k} - r^{*} \| + \varepsilon \tau \| R_{U}^{-1} B^{*} r^{k}_{\delta} \| \\
&\leq \gamma \| e^{k}_u \| + \tau M^{2} \delta \| e^{k}_u \| + \varepsilon \tau M \| r^{k}_{\delta} - r^{*} \| \quad \text{(by \cref{eq:I})} \\
&\leq \gamma \| e^{k}_u \| + \tau M^{2} \delta \| e^{k}_u \| + \varepsilon \tau M^{2} (1+\delta) \| e^{k}_u \| \quad \text{(by \cref{eq:II})}.
\end{align*}    
\[
\implies \| e^{k+1}_u \| \leq \big( \gamma + \tau M^{2} \big( \delta + \varepsilon (1+\delta) \big) \big) \| e^{k}_u \|.
\]  
By assumption,  
\(
\delta + \varepsilon (1+\delta) < \frac{1 - \gamma}{\tau M^{2}},
\) 
which ensures that, \(\| e^{k+1}_u \| \leq \| e^{k}_u \|\). Hence we conclude that $u^{k}_{\delta,\varepsilon} \to u^* \in U$, and by \cref{eq:II}, $r^{k}_{\delta} \to r^* \in V$. 
\end{proof}

In the Uzawa Double Deep Ritz framework described in \cref{alg:uzawa_DRR}, each of the two inner
minimization problems (the $r$ update and the $u$ update) is not exactly solved. Instead, assume both are approximated using a few steps of a gradient descent method  applied to the corresponding Ritz energies.

A key observation is that even when the gradient iterations are very few, or when the corresponding  Ritz energies are far from being fully minimized, the overall Uzawa iteration still converges. The essential requirement is that each approximate inner update for $r$ and $u$ moves in the correct descent direction. Under this directional correctness, the Uzawa outer loop remains convergent.

To formalize this idea, we analyze a simplified setting where each inner Ritz problem is approximated by a single gradient step,
\begin{equation}\label{eq:one_step_grad}
\begin{cases}
   1.~ r^k = r^{k-1} - \alpha \left( r^{k-1} - R_V^{-1}(\ell - B u^k) \right)\\
        2.~ u^{k+1} = u^k - \omega \left( u^{k+1} - u^k - \tau R_U^{-1} B^* r^k \right) 
\end{cases}
\end{equation}
This iteration coincides structurally with the classical Arrow--Hurwicz method \cite{benzi2005numerical}, thereby establishing a connection between the Deep Ritz framework and classical saddle-point algorithms.

Recall that $U$ and $V$ are Hilbert spaces with Riesz maps
$R_U : U \to U^\ast$ and $R_V : V \to V^\ast$.  
Let $B : U \to V^\ast$ be a bounded linear operator satisfying the
continuous inf-sup bounds as in \cref{bound_B}. The convergence theorem below is established by demonstrating that the error satisfies a contraction property governed by an associated iteration matrix.
The structure of the proof is inspired by classical arguments from matrix iterative analysis, as developed in \cite{horn2012matrix} and \cite{varga1999matrix}.
\begin{theorem}\label{thm:uddr}
Consider the inexact Uzawa Double Deep Ritz iteration in which each Deep Ritz
problem is approximated by a single gradient step \cref{eq:one_step_grad} and $(r^\ast,u^\ast)$ be the saddle point solution of \cref{mixed problem_cont}. Then for \(
0 < \alpha < 1,~
0<\omega<1,~
0 < \tau < \frac{2}{M^2}
\), the error, 
\(
e_r^k = r^k - r^\ast \longrightarrow 0,~
e_u^k = u^k - u^\ast\longrightarrow 0~\text{as } k \to \infty. \)
\end{theorem}
\begin{proof}
 Assume that
\(
0 < \alpha < 1,~
0 < \omega < 1,~
0 < \tau < \frac{2}{M^{2}}
\)
and recall the one–step gradient iteration \cref{eq:one_step_grad}.
Consider first the update for $r^k$;
 \(
r^{k}
= r^{k-1} - \alpha \left( r^{k-1} - R_V^{-1}(\ell - B u^{k}) \right),
\)
then,
\begin{equation}
\label{eq:error_r}
e_r^{k} = (1 - \alpha) e_r^{k-1} - \alpha R_V^{-1} B\, e_u^{k}.
\end{equation}
Next, consider the update for $u^{k+1}$;
\(
u^{k+1}
= u^{k} - \omega\left( u^{k+1} - u^{k} - \tau R_U^{-1} B^{*} r^{k} \right).
\)
Rearranging,
\[
(1+\omega) u^{k+1} = (1+\omega) u^{k} + \tau \omega R_U^{-1} B^{*} r^{k},
\]
and hence,
\[
u^{k+1} = u^{k} + \tilde{\tau} R_U^{-1} B^{*} r^{k},
\qquad 
\tilde{\tau} := \frac{\tau \omega}{1+\omega},
\]
then we get,
\(
e_u^{k+1}
= e_u^{k} + \tilde{\tau} R_U^{-1} B^{*} e_r^{k}.
\)

\noindent Using \cref{eq:error_r}, this becomes
\(
e_u^{k+1}
= e_u^{k}
+\tilde{\tau} R_U^{-1} B^{*}\left( (1 - \alpha) e_r^{k-1}
- \alpha R_V^{-1} B e_u^{k} \right),\)
\begin{equation}\label{error_u}
  e_u^{k+1}
= \left( I - \tilde{\tau}\alpha R_U^{-1} B^{*} R_V^{-1} B \right)e_u^{k}
+ \tilde{\tau}(1 - \alpha) R_U^{-1} B^{*} e_r^{k-1}.  
\end{equation}
Combining \cref{eq:error_r} and \cref{error_u}, we write the coupled error system as
\[
\begin{bmatrix}
e_r^{k} \\[4pt]
e_u^{k+1}
\end{bmatrix}
=
\begin{bmatrix}
(1-\alpha)I_V & -\alpha R_V^{-1} B \\[4pt]
\tilde{\tau}(1-\alpha) R_U^{-1} B^{*} & I_U - \tilde{\tau}\alpha R_U^{-1} B^{*} R_V^{-1} B
\end{bmatrix}
\begin{bmatrix}
e_r^{k-1} \\[4pt]
e_u^{k}
\end{bmatrix},
\]
where $I_V$ and $I_U$ denote identity operators on $V$ and $U$.
Let
\begin{equation}\label{eq:matrix1}
E^{k}
=\begin{bmatrix} e_r^{k} \\ e_u^{k+1} \end{bmatrix},
\qquad
\mathbb{A}
=
\begin{bmatrix}
(1-\alpha)I_V & -\alpha R_V^{-1} B \\[4pt]
\tilde{\tau}(1-\alpha)R_U^{-1} B^{*} &
I_U - \tilde{\tau}\alpha R_U^{-1} B^{*} R_V^{-1} B
\end{bmatrix},
\end{equation}
so that
\begin{equation}
\label{eq:contractive_map}
E^{k} = \mathbb{A} E^{k-1}.
\end{equation}
The goal is to show that the spectral radius is less than unity, i.e.,
\(
\rho(\mathbb{A}) < 1,
\)
which guarantees that $\mathbb{A}^{k} \to 0$ and therefore
\(
E^{k} \to 0,
\)
establishing convergence of the iteration.
Let
\(
G := R_U^{-1} B^{*} R_V^{-1} B : U \to U,
\)
which is symmetric positive definite with spectrum contained in $[m^{2}, M^{2}]$.
Consider an eigen pair $(v,w) \in V \times U$,
\[
\mathbb{A}
\begin{pmatrix} v \\ w \end{pmatrix}
= \lambda
\begin{pmatrix} v \\ w \end{pmatrix}.
\]
From the first block row,
\begin{equation}
\label{eq:v_relation}
(1-\alpha-\lambda)v = \alpha R_V^{-1} B w,
\qquad 
v = \frac{\alpha}{1-\alpha-\lambda} R_V^{-1} B w, \qquad 1-\alpha-\lambda \neq 0.
\end{equation}
Substituting \cref{eq:v_relation} into the second block row yields
\[
(1-\alpha-\lambda)(I_U - \alpha\tilde{\tau} G - \lambda I_U) w
= -\tilde{\tau}\alpha(1-\alpha) G w.
\]
Let $w$ be an eigenvector of $G$ with eigenvalue $\mu\in[m^{2},M^{2}]$ where $w\neq 0$. On substituting this, we get the scalar characteristic equation, \(
(1-\alpha-\lambda)(1-\lambda)
+ \lambda\tilde{\tau}\alpha\mu = 0.
\)
Thus, the characteristic polynomial is
\begin{equation}\label{charc_eq}
    p(\lambda)
= \lambda^{2}
- \left( 2 - \alpha - \tilde{\tau}\alpha\mu \right)\lambda
+ (1-\alpha).
\end{equation}
To show that all eigenvalues satisfy $|\lambda|<1$, we use the Schur-Cohn Test in \cite{anderson1973simplified}, which states that  a general quadratic polynomial $p(\lambda)$ with real coefficients  has modulus of  both roots (real or complex conjugates) strictly less than 1 if and only if $p(1)>0,~p(-1)>0$ and $|\text{constant term}|<|\text{leading coefficient}|$.

 For the characteristic polynomial \cref{charc_eq} we have,
\[
p(1)
= 1 - (2-\alpha-\tilde{\tau}\alpha\mu) + 1 - \alpha
= \tilde{\tau}\alpha\mu > 0
\]
\[
p(-1)
= 1 + (2-\alpha-\tilde{\tau}\alpha\mu) + 1 - \alpha
= 4 - 2\alpha - \tilde{\tau}\alpha\mu.
\]
Since $\mu\le M^{2}$ and $\tilde{\tau} < \frac{2}{M^{2}}$, we get
\(
-\tilde{\tau}\alpha\mu > -2\alpha,
\)
so
\[
p(-1) > 4 - 2\alpha - 2\alpha = 4(1-\alpha) > 0.\]
Moreover, $|1-\alpha|<1$ since $\alpha \in (0,1)$. Therefore, for our assumptions we get $|\lambda_1|,|\lambda_2|<1 \implies \rho(\mathbb{A}) < 1. $

 Since $\rho(\mathbb{A})<1$, it follows that $\mathbb{A}^{k}\to 0$, and hence by
\cref{eq:contractive_map},
$E^{k} \to 0$,
i.e.,
$e_r^{k} \to 0$ and $e_u^{k} \to 0$.
Therefore,
$r^{k} \to r^{*}$ and 
$u^{k} \to u^{*}$,
which establishes convergence of the iterative scheme.
\end{proof}
\begin{remark}
In particular, the method converges even though the Ritz energies are
not minimized exactly.  The only requirements are that each
gradient step moves in the correct descent direction, and that the Uzawa
stepsize satisfies $0 < \tau < \frac{2}{M^2}$.
\end{remark} 
The previous result shows that even a single gradient step per update is sufficient for convergence. In practice, one may perform multiple gradient steps in each update to improve accuracy. This leads to a generalized iteration in which the residual update is approximated using $n$ gradient steps and the solution update using $m$ gradient steps. This extends the analysis in \cite{badea2025convergence}, where only one of the updates is treated with multiple inner iterations. In contrast, our framework allows both updates to be treated in a multilevel manner.

\begin{theorem}[Convergence with $n$ and $m$ inner gradient steps]\label{thm:nmconvg}
Consider the inexact Uzawa Double Deep Ritz iteration in which the first Deep Ritz problem is approximated by $n$ gradient steps
and the second by $m$ gradient steps and $(r^\ast,u^\ast)$ be the saddle point solution of \cref{mixed problem_cont}. Then for \(
0 < \alpha < 1,~
0<\omega<1,~
0 < \tau < \frac{2}{M^2},
\) the error,
\(
e_r^k = r^k - r^\ast \longrightarrow 0,~
e_u^k = u^k - u^\ast\longrightarrow 0
~\text{as } k \to \infty. \)

\end{theorem}
\begin{proof}
Following the multi-step gradient updates for $n$ and $m$ inner iterations, the error terms yield the recurrence relations
\begin{equation}\label{mult_grad1}
e_r^k = (1-\alpha)^n e_r^{k-1} - \left(1-(1-\alpha)^n\right) R_V^{-1}B e_u^k,
\end{equation}
and
\begin{equation}\label{mult_grad2}
e_u^{k+1} = e_u^k + \left(1-(1-\omega^*)^m\right) \tau R_U^{-1}B^\ast e_r^k,
\end{equation}
where $\omega^*=\frac{\omega}{1+\omega}$. Substituting \cref{mult_grad1} into \cref{mult_grad2} yields the coupled error system $E^k = \mathbb{A}E^{k-1}$ with the iteration matrix
\begin{equation}\label{eq:matrix2}
\mathbb{A}=
\begin{bmatrix}
(1-\alpha)^n I_V
&
-\left(1-(1-\alpha)^n\right)R_V^{-1}B
\\[6pt]
\tau(1-\alpha)^n(1-(1-\omega^*)^m)R_U^{-1}B^\ast
&
I_U-\tau(1-(1-\omega^*)^m)(1-(1-\alpha)^n)G
\end{bmatrix}.
\end{equation}
By executing a block-row elimination analogous to the proof of \cref{thm:uddr}, the eigenvalues $\lambda$ of $\mathbb{A}$ associated with an eigenvector $w$ of $G$ (with eigenvalue $\mu \in [m^2, M^2]$) satisfy the scalar characteristic equation
\begin{equation}\label{charc_eq2}
p(\lambda) = \lambda^2 - \lambda \Big( 1+(1-\alpha)^n - \tau\mu(1-(1-\omega^*)^m)(1-(1-\alpha)^n) \Big) + (1-\alpha)^n = 0.
\end{equation}
To establish $\rho(\mathbb{A}) < 1$, we apply the Schur-Cohn test. Direct evaluation yields
\begin{align*}
p(1) &= \tau\mu(1-(1-\omega^*)^m)(1-(1-\alpha)^n) > 0, \\
p(-1) &= 2\left( 2(1-\alpha)^n + (1-\omega^*)^m(1-(1-\alpha)^n) \right) > 0.
\end{align*}
Given that $\alpha \in (0,1)$, the constant term satisfies $|(1-\alpha)^n| < 1$. Thus, all roots satisfy $|\lambda| < 1$, ensuring $\rho(\mathbb{A}) < 1$, which concludes the proof of convergence.
\end{proof}
\begin{remark}
The additional inner gradient steps modify the spectral structure of the iteration matrix in a way that may reduce the spectral radius relative to the single-step scheme. This suggests potential acceleration per outer (Uzawa) iteration, depending on the choice of n and m.
\end{remark}
\section{Neural Network Approximation and Training Algorithm}\label{sec:Training-alg}
In this section, we present the neural network implementation of the Uzawa Double Deep Ritz method and describe the associated training strategies used to approximate the associated minimization problems.

In the Uzawa Double Deep Ritz framework, each outer Uzawa iteration consists of two Ritz minimization problems: one corresponding to the residual update and the other to the solution update. To make the parameterization explicit, we represent the residual approximation using a single hidden layer neural network with $n$ neurons as,
\[
    r_\eta(x) = r_{(c,b)}(x) = \sum _{i=1}^{n}c_i\Phi(x+b_i),
\]
where the vector $b=[b_1, \dots, b_n]^T$ acts as a shift inside the nonlinear activation function $\Phi$, and the vector $c=[c_1, \dots, c_n]^T$ linearly weights the hidden units. Similarly, the solution approximation using a single hidden layer neural network with m neurons is given by
\[
    u_\theta(x) = u_{(d,\beta)}(x)=\sum_{i=1}^{m}d_i\Psi(x+\beta_i),
\]
where the vector $\beta=[\beta_1, \dots, \beta_m]^T$ acts as a shift inside the nonlinear activation function $\Psi$, and the vector $d=[d_1, \dots, d_m]^T$ linearly weights the hidden units. The complete parameter sets are thus \(\eta = (c,b)\) and \(\theta = (d,\beta)\).

The corresponding Deep Ritz energy functionals are given by
 \[\mathcal{J}_{r_\eta(c,b)}=\frac{1}{2}\| r_{\eta(c,b)} \|^{2}_V
            - \ell(r_{\eta(c,b)})
            + b\!\left( u_{\theta}^{k}, r_{{\eta}(c,b)} \right)\] and \[\mathcal{J}_{u_\theta(d,\beta)}=\frac{1}{2}\| u_{\theta(d,\beta)} \|^{2}_U
            - \tau b\!\left( u_{\theta(d,\beta)}, r_{\eta}^{k} \right)
            - \big( u_{\theta}^{k}, u_{\theta(d,\beta)} \big)_U\]
At each Uzawa iteration \(k\), these two minimization problems are solved approximately. In this work, we consider two distinct training strategies for performing these inexact updates.
\subsection{Block Gradient Training}
In the first approach, we employ a block-gradient strategy to approximate each minimizer. This method is analogous to the gradient-block approach introduced in~\cite{he2018relu}, where the outer and inner parameters are updated in an alternating fashion. More precisely, for each minimization problem, we first update the outer parameters by solving exactly while keeping the inner parameters fixed. Subsequently, the inner parameters are updated using gradient descent based on the updated outer parameter. This strategy can be interpreted as performing an exact minimization with respect to the linear coefficients while retaining an inexact (gradient descent) update for the nonlinear parameters, thereby aligning naturally with the inexact Uzawa framework.

The resulting training procedure consists of an outer Uzawa loop coupled with two inner loops corresponding to the residual and solution updates. The full structure is summarized in \cref{alg:block-training}.
\begin{algorithm}[htbp]
\caption{Uzawa Double Deep Ritz Training Algorithm (Block-Gradient)}
\label{alg:block-training}
\begin{algorithmic}[1]
\REQUIRE initial parameters $\eta^{0}=(c^{0},b^{0})$, $\theta^{0}=(d^{0},\beta^{0})$; learning rates $\alpha,\omega>0$; inner counts $N_{r},N_{u}$; step size $\tau>0$; outer iterations $N$
\FOR{$k=0,1,\dots,N-1$}
    \STATE $(c,b)\gets \eta^{k}$
    \begin{tikzpicture}[overlay, remember picture]
     \node[anchor=east] at (8,-0.9) 
    {\(\left.\begin{array}{c} \\ \\ \\ \end{array}\right\} \text{Deep Ritz 1}\)};\end{tikzpicture}
    \FOR{$i=1,\dots,N_{r}$}
        \STATE $c\gets \operatorname*{\arg\min}_{c}\,
               \mathcal{J}_{r_{\eta}}(c,b;\theta^{k})$
        \STATE $b\gets b-\omega\,\nabla_{b}\mathcal{J}_{r_{\eta}}(c,b;\theta^{k})$
    \ENDFOR
    \STATE $\eta^{k+1}\gets (c,b)$
    \STATE $(d,\beta)\gets \theta^{k}$
     \begin{tikzpicture}[overlay, remember picture] \node[anchor=east] at (8,-1.1) 
    {\(\left.\begin{array}{c} \\ \\ \\ \end{array}\right\} \text{Deep Ritz 2}\)}; 
    \end{tikzpicture}
    \FOR{$j=1,\dots,N_{u}$}
    \STATE $d\gets \operatorname*{\arg\min}_{d}\,
    \mathcal{J}_{u_{\theta}}(d,\beta;\eta^{k+1},\theta^{k})$
    \STATE $\beta\gets \beta-\alpha\,\nabla_{\beta}\mathcal{J}_{u_{\theta}}(d,\beta;\eta^{k+1},\theta^{k})$
    \ENDFOR
    \STATE $\theta^{k+1}\gets (d,\beta)$
\ENDFOR
\RETURN $\theta^{N},\eta^N$
\end{algorithmic}
\end{algorithm}
\subsection{Full Gradient Training}
As an alternative to the block-gradient strategy, we consider a full gradient descent approach in which both the outer and inner parameters are updated simultaneously using gradient-based optimization. In this setting, no parameter block is solved exactly. Instead, all parameters are updated using gradient descent applied directly to the corresponding Deep Ritz functionals. The resulting training procedure modifies \cref{alg:block-training} by replacing the exact solves in lines 4 and 10 with standard gradient steps: \(c \gets c-\omega  \nabla_{c} \mathcal{J}_{r_\eta}(c,b; \theta^k)\) and \(d\gets d-\alpha\nabla_{d} \mathcal{J}_{u_\theta}(d,\beta;\eta^{k+1},\theta^{k})\) respectively. All other initialization steps and loop structures remain identical to \cref{alg:block-training}.

The two training strategies described above represent different approaches to approximating the inexact Uzawa updates. The block-gradient method leverages partial exact minimization and is expected to produce more stable updates even with a small number of inner iterations. In contrast, the full gradient approach relies entirely on iterative optimization and may require a larger number of inner iterations to achieve comparable accuracy.

In the next section, we investigate these differences through numerical experiments. In particular, we demonstrate that the block-gradient strategy exhibits robustness with respect to the number of inner iterations, whereas the performance of the full gradient method is significantly influenced by this parameter. 
\begin{remark}
    We expect the convergence result established in \cref{thm:uddr} and \cref{thm:nmconvg} for the continuous formulation to carry over to this neural network-based framework, and the numerical experiments in the following section provide evidence supporting this behavior.
\end{remark}
\section{Numerical Experiments}\label{sec:num_exp}
In this section, we evaluate the performance and convergence behavior of Uzawa Double Deep Ritz method applied to a 1D transport problem and a 2D Poisson problem. The proposed framework relies on an alternating iteration that successively minimizes the Deep Ritz energy functionals corresponding to the residual and solution updates. To optimize these energy functionals, we evaluate two distinct minimization strategies introduced in \Cref{sec:Training-alg}: the block- gradient training (detailed in \cref{alg:block-training}) and its full gradient variant (the gradient step modification of \cref{alg:block-training}).
\subsection{Example 1 (1D Transport Problem)}
Consider the one-dimensional \\boundary value problem
\begin{equation*}
\left\{
\begin{array}{r@{}l}
u'(x) &{}= f(x), \quad x \in (0,1), \\[2pt]
u(0) &{}= 0,
\end{array}
\right.
\end{equation*}
with variational formulation: Find $u \in L^2(0,1)$ such that:
\begin{equation*}
-\int_0^1 u(x) v'(x) \, dx
=
\int_0^1 f(x) v(x) \, dx\quad \forall\, v \in H^1_{0)}(0,1) = \{ v \in H^1(0,1) : v(1)=0 \}.
\end{equation*}
We equip the spaces with the following inner products:
 For the trial space $U=L^2(0,1)$,
\(
(u,w)_U = \int_0^1 u(x) w(x) \, dx\), and
for the test space $V=H^1_{0)}(0,1)$,
\(
(v,w)_V = \int_0^1 v'(x)w'(x) \, dx .
\)

 To construct the Uzawa Double Deep Ritz formulation, we approximate  \(r\) and  \(u\) using one hidden layer neural networks for which exact integration and differentiation are employed. In particular, we take  
\[
r_\eta(x)=\sum_{i=1}^{n} c_i\, \mathrm{ReLU}(b_i-x), 
\qquad
u_\theta(x)=\sum_{i=1}^{m} d_i\, H(x-\beta_i),
\]
with trainable parameters 
\(\eta=(c_i,b_i)_{i=1}^{n}\), \(\theta=(d_i,\beta_i)_{i=1}^{m}\) and \(H(x)\) represents Heaviside step function.
This architectural choice mirrors the regularity structure of the underlying variational problem: the ReLU basis yields piecewise linear approximations of \(r\), whereas the Heaviside basis produces piecewise constant approximations of \(u\), in accordance with the double Ritz characterization of the Uzawa update. This choice is also computationally advantageous, as both integration and differentiation can be carried out exactly for these basis functions.
\subsubsection{Block-Gradient Training}
We use the Block-Gradient Training mentioned in \cref{alg:block-training}  with particular parameters: Uzawa outer iterations: \(25\), inner iterations: $N_r=N_u=1$, Uzawa step size \(\tau=0.5\), learning rates \(\alpha=0.04\) and \(\omega=0.01\), \(20\) neurons ($n=m=20$), Source term \(f(x)=1\).

The numerical results exhibit clear convergence consistent with the theoretical predictions. The evolution of the residual approximation \(r_\eta^{k}\) across outer iterations $k$ is shown in  Figure~\ref{fig:residue}, which shows its convergence towards zero.  The shapes confirm the decay of the residual and illustrate the contraction properties of the Uzawa update.
\begin{figure}[htbp]
    \centering
    \begin{subfigure}[b]{0.41\textwidth}
    \centering
    \includegraphics[width=\textwidth]{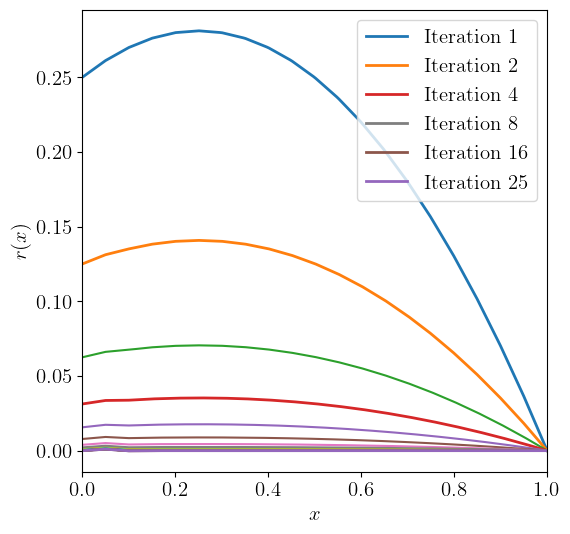}
    \subcaption{}
    \label{fig:residue}
    \end{subfigure}
    \begin{subfigure}[b]{0.4\textwidth}
    \centering
    \includegraphics[width=\textwidth]{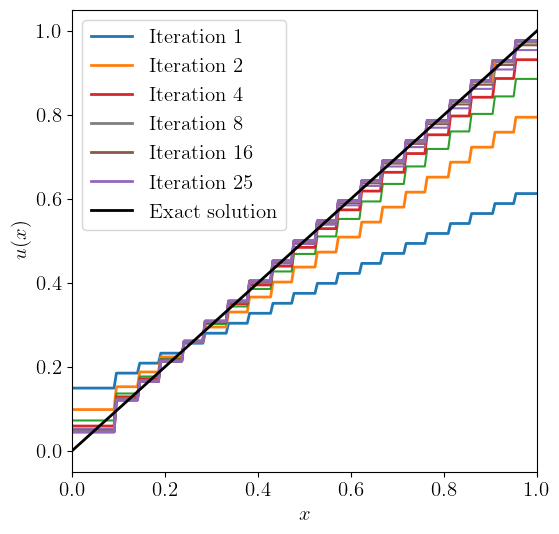}
    \subcaption{}
    \label{fig:u_approx}
    \end{subfigure}
    \caption{(a) Residual $r_\eta(x)$ plotted at Uzawa outer iteration. (b) Convergence towards the exact solution of $u_\theta^{k}$ over successive Uzawa iterations.}
\end{figure}
The successive approximations \(u_\theta^{k}\) also move steadily toward the exact solution as shown in Figure \ref{fig:u_approx}. 

 To assess the performance of the proposed Uzawa Double Deep Ritz method, we monitor two quantities across the outer Uzawa iteration, which are natural quantities in the minimal residual framework. \begin{itemize}
    \item \textit{Residual norm:} measures how well the PDE is satisfied.
    \item \textit{Solution error:} measures how close we are to the exact solution.
\end{itemize}
As predicted by the minimal residual framework, the decay of the residual norm is directly reflected in the reduction of the solution error as in Figure \ref{fig:res,sltn}. Although the two quantities do not overlap, they exhibit proportional behavior, indicating that the residual provides a reliable measure of the approximation error.

Figure~\ref{subfig:bg_u_er} and \ref{subfig:bg_r_er} illustrate the effect of varying the inner iterations $(N_r, N_u)$ on the convergence behavior. As seen in Figure~\ref{subfig:bg_u_er}, the solution error $\|u - u^{k}_\theta\|_U$ follows nearly identical trajectories for all configurations, indicating that the final approximation is largely insensitive to the number of inner iterations. In contrast, Figure~\ref{subfig:bg_r_er} shows noticeable variation in the residual norm $\|r^{k}_\eta\|_V$, reflecting differences in the accuracy of the inner Ritz steps. This indicates that the Uzawa iteration does not require highly accurate inner minimization; rather, it suffices that the updates provide suitable descent directions. This observation is consistent with the inexact Uzawa \cref{thm:uddr}, where convergence depends on descent properties rather than exact solves. 
\begin{figure}[htbp]
    \centering
    \begin{subfigure}[b]{0.32\textwidth}
    \centering
    \includegraphics[width=\textwidth]{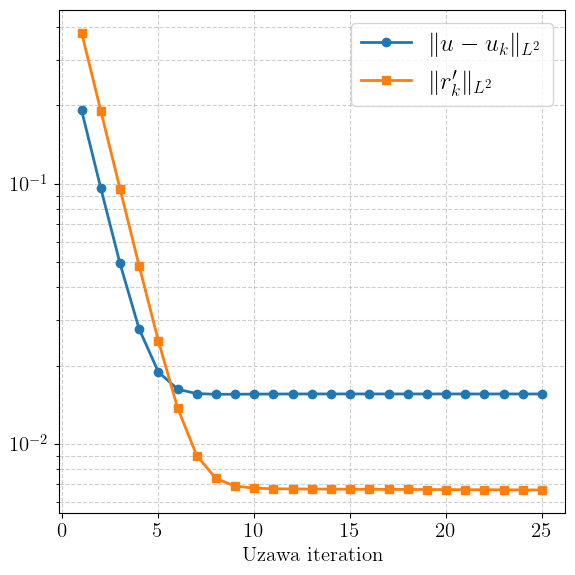}
    \subcaption{}
    \label{fig:res,sltn}
    \end{subfigure}
    \begin{subfigure}[b]{0.32\textwidth}
        \includegraphics[width=\linewidth]{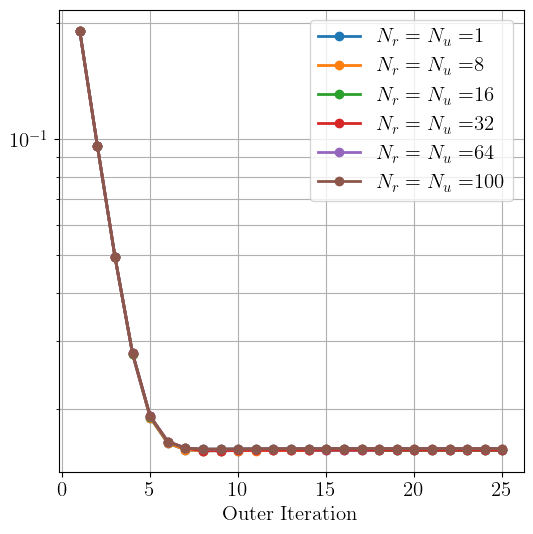}
         \subcaption{}
         \label{subfig:bg_u_er}
    \end{subfigure}
    \begin{subfigure}[b]{0.32\textwidth}
        \includegraphics[width=\linewidth]{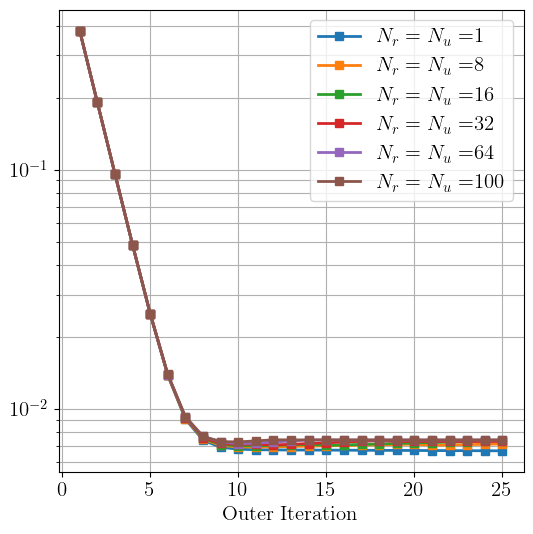}
        \subcaption{}
        \label{subfig:bg_r_er}
    \end{subfigure}
    \caption{(a) Evolution of the residual norm and solution error over Uzawa iterations. (b) Effect of inner iteration $(N_r, N_u)$ on solution error $\|u - u^{k}_\theta\|_U$, showing nearly identical convergence across configurations. (c) Effect of inner iteration $(N_r, N_u)$ on residual norm $\|r^{k}_\eta\|_V$, showing sensitivity to inner solve accuracy.}
    \label{fig:inner_iter}
\end{figure}
 \begin{figure}[htbp]
\centering

\begin{subfigure}{0.32\textwidth}
   \includegraphics[width=\textwidth]{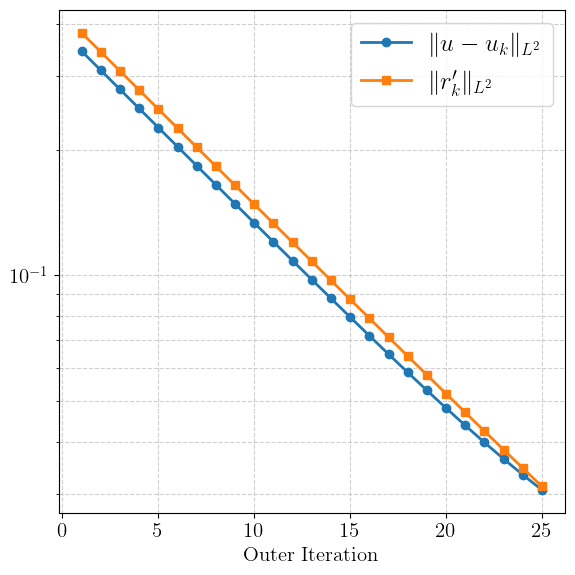}
   \caption{$\tau=0.1$}
\end{subfigure}
\hfill
\begin{subfigure}{0.32\textwidth}
   \includegraphics[width=\textwidth]{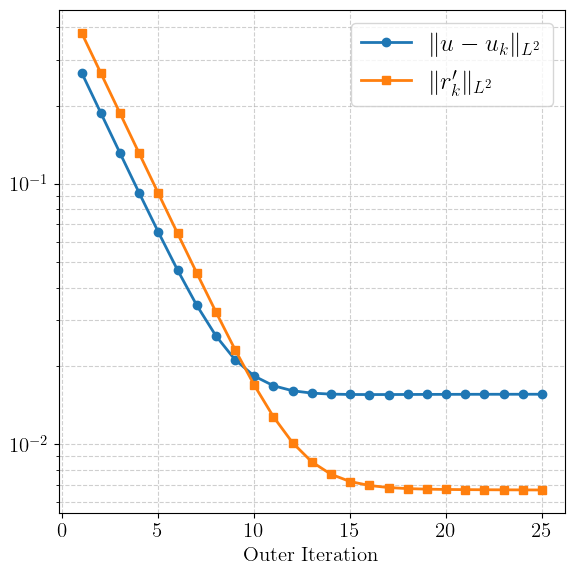}
   \caption{$\tau=0.3$}
\end{subfigure}
\hfill
\begin{subfigure}{0.32\textwidth}
   \includegraphics[width=\textwidth]{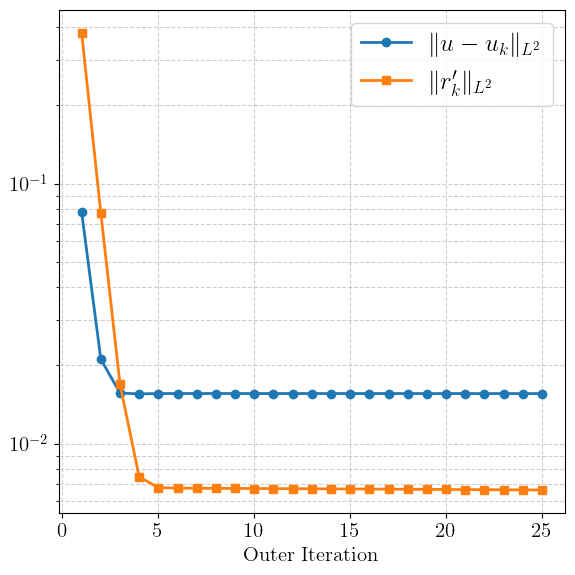}
   \caption{$\tau=0.8$}
\end{subfigure}

\vspace{0.5cm}

\begin{subfigure}{0.32\textwidth}
   \includegraphics[width=\textwidth]{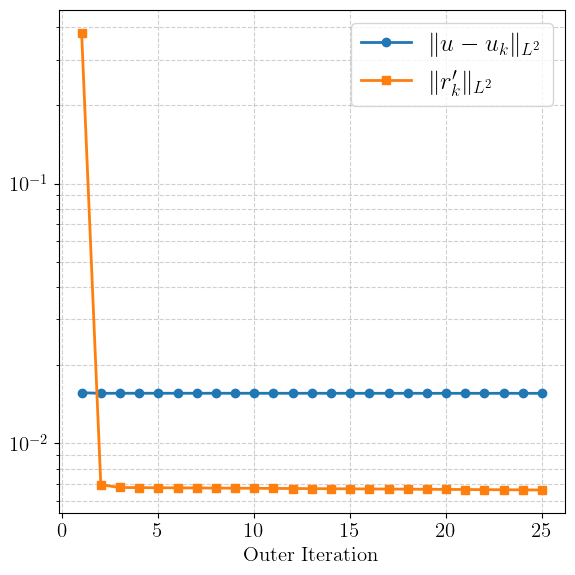}
   \caption{$\tau=1$}
\end{subfigure}
\hfill
\begin{subfigure}{0.32\textwidth}
   \includegraphics[width=\textwidth]{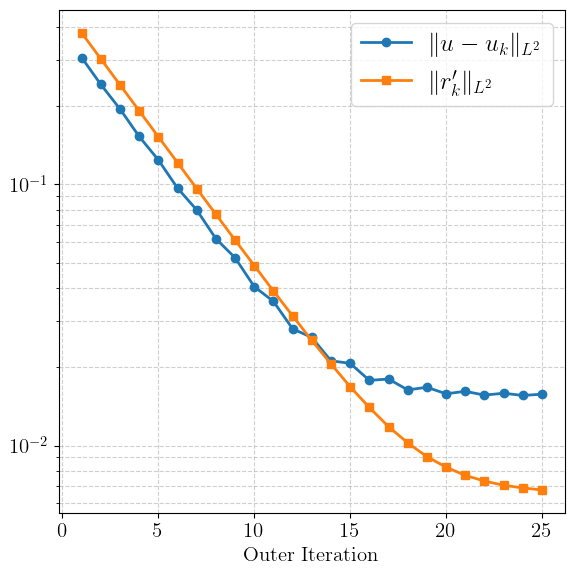}
   \caption{$\tau=1.8$}
\end{subfigure}
\hfill
\begin{subfigure}{0.32\textwidth}
   \includegraphics[width=\textwidth]{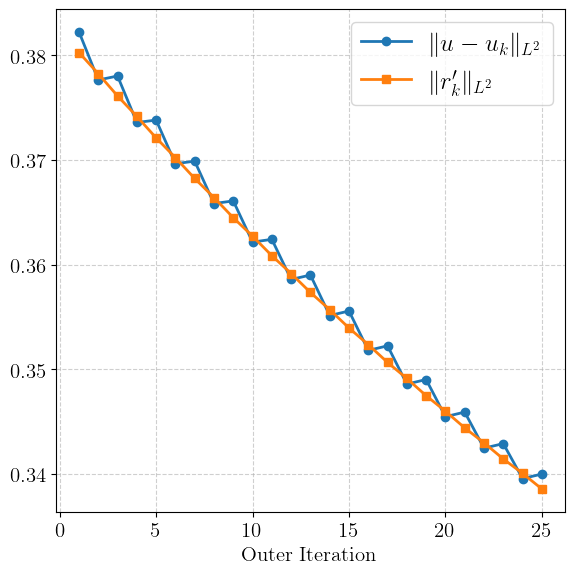}
   \caption{$\tau=2$}
\end{subfigure}

\caption{Influence of the Uzawa step size $\tau$ on convergence.}
\label{fig:converg_tau}
\end{figure}

The numerical results consistently show that, although the inner Deep Ritz routines are not solved to full convergence, they still produce descent directions with sufficient accuracy to preserve the stability of the outer Uzawa iteration (in line with \cref{thm:uddr}). This inexactness does not degrade the overall performance. On the contrary, the iterates associated with the $u$-variable exhibit a clear contractive behavior, consistent with the theoretical results. The observed decay of the corresponding energy over successive iterations further supports the robustness and effectiveness of the proposed Uzawa Double Deep Ritz framework.

The results in Figure \ref{fig:converg_tau} confirm the theoretical stability condition  \(0<\tau<2/M^2\) showing the effect of different $\tau$ values on convergence. For small values of \(\tau\), convergence is slow due to under-relaxation, while values near the upper bound yield faster contraction. Larger values lead to instability, consistent with the theoretical prediction.

\subsubsection{Full Gradient Training}
We now investigate the performance of the full gradient training strategy. In contrast to the block-gradient method, both outer and inner parameters are updated using gradient descent. We consider the same setting as in the previous case, but now with 50 Uzawa outer iterations and $N_r=N_u=1024$ inner iterations.

Since the residual and solution evolution plots exhibit qualitatively similar behavior to those of the block-gradient case, we report only the convergence metrics in Figure~\ref{fig:FG_}, which more clearly illustrate the differences between the two training strategies.

 Figure~\ref{subfig:fg_u_er} and \ref{subfig:fg_r_er} show the effect of varying the number of inner iterations $(N_r, N_u)$ in the full gradient training strategy. Here, both the solution error and the residual norm exhibit a strong dependence on the number of inner iterations, in contrast to the block-gradient case. As seen in Figure~\ref{subfig:fg_u_er}, the solution error $\|u - u^{k}_\theta\|_U$ improves significantly as the number of inner iterations increases. Similarly, Figure~\ref{subfig:fg_r_er} shows that the residual norm $\|r^{k}_{\eta}\|_V$ is sensitive to $(N_r, N_u)$, with larger values yielding faster and more accurate convergence.

 This behavior highlights a key distinction from the block-gradient approach: in the full gradient method, accurate minimization of the intermediate Ritz functionals is crucial for achieving good convergence. In particular, a small number of inner iterations results in slower convergence and higher error levels, demonstrating the sensitivity of the full gradient scheme to the accuracy of inner solves.
\begin{figure}[htbp]
    \centering
    \begin{subfigure}[b]{0.32\textwidth}
    \centering
    \includegraphics[width=\linewidth]{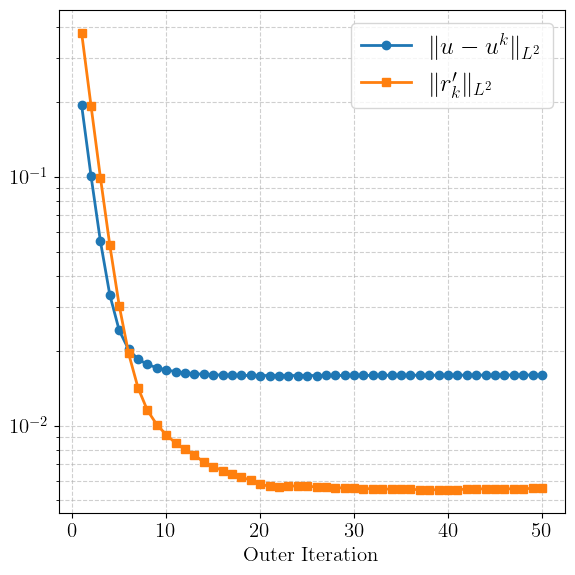}
    \subcaption{}
    \label{fig:FG_}
    \end{subfigure}
    \begin{subfigure}[b]{0.32\textwidth}
        \includegraphics[width=\linewidth]{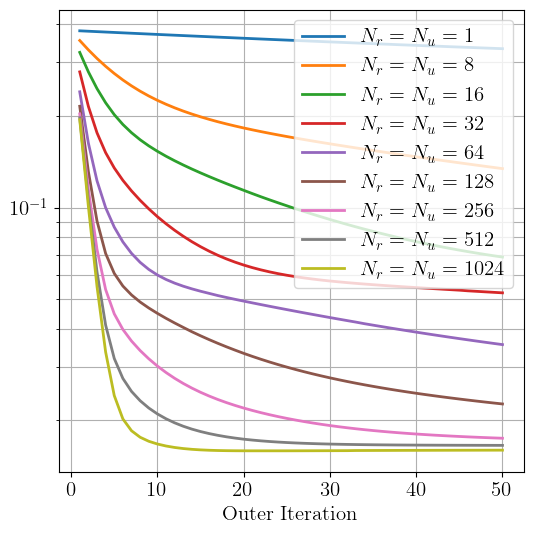}
         \subcaption{}
         \label{subfig:fg_u_er}
    \end{subfigure}
    \begin{subfigure}[b]{0.32\textwidth}
        \includegraphics[width=\linewidth]{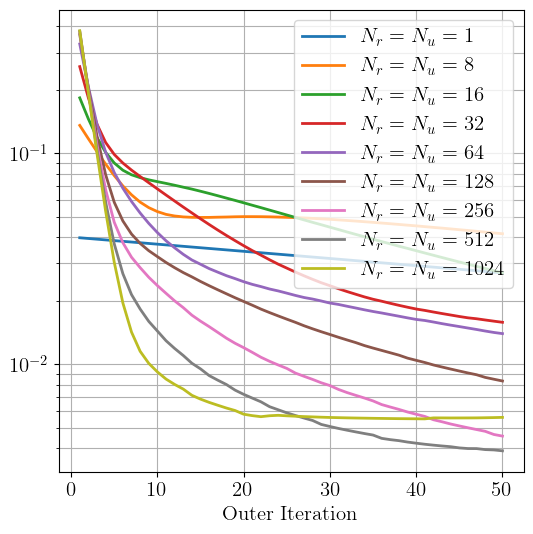}
        \subcaption{}
        \label{subfig:fg_r_er}
    \end{subfigure}
    \caption{(a) Evolution of the residual norm and solution error over Uzawa iterations. 
(b) Effect of inner iteration counts $(N_r, N_u)$ on the evolution of the solution error $\|u - u^{k}_\theta\|_U$. (c) Effect of inner iteration counts $(N_r, N_u)$ on the evolution of the residual norm $\|r^{k}_\eta\|_V$. }
    \label{fig:full_inner}
    
\end{figure}

 These results demonstrate that the block-gradient training strategy is significantly more robust with respect to the accuracy of inner solves. In particular, accurate solution approximation can be achieved even when the residual minimization is performed only approximately. In contrast, the full gradient approach requires sufficiently accurate inner optimization to ensure convergence, making it more sensitive to algorithmic parameters. 

\subsection{Example 2 (2D Poisson Problem)}
Let us consider the following two-dimensional boundary value problem
\begin{equation*}
\left\{
\begin{array}{r@{}ll}
-\Delta u(x,y) &{}= f(x,y), & (x ,y)\in \Omega= (0,1) \times(0,1), \\[2pt]
u(x,y) &{}= 0, & (x,y) \in \partial \Omega,
\end{array}
\right.
\end{equation*}
with variational formulation: Find $u \in H^1_{0}(\Omega) $ such that:
\begin{equation*}
\int_{\Omega} \nabla u \nabla v \, dx
=
\int_{\Omega} fv \, dx\quad \forall\, v \in H^1_{0}(\Omega) = \{ v \in H^1(\Omega) : v=0~ \text{on}~\partial\Omega \}.
\end{equation*}
We equip the trial and test space $U=V=H^1_{0}(\Omega)$, with the following inner product:
\begin{equation*}
(u,w)_U = \int_{\Omega} \nabla u \nabla w \, dx .
\end{equation*}
To construct the Uzawa Double Deep Ritz methodology for the two-dimensional Poisson problem, we approximate both the solution variable $u$ and the residual variable $r$ using single hidden layer fully connected feedforward neural networks. Specifically, we consider the parametric representations:
\[
u_\theta(x,y) = \sum_{i=1}^{m} d_i \tanh(a_{i1}x + a_{i2}y + \beta_i),
\qquad
r_\eta(x,y) = \sum_{i=1}^{n} c_i \tanh(w_{i1}x + w_{i2}y + b_i),
\]
where \(u_\theta, r_\eta : \mathbb{R}^2 \to \mathbb{R}\) are feedforward neural networks with a single hidden layer, \(m\) and \(n\) neurons, respectively, and hyperbolic tangent activation.\\ 
Here \(w_{i1}, w_{i2}, a_{i1}, a_{i2} \in      \mathbb{R}\) denote input-layer weights, \(b_i, \beta_i \in \mathbb{R}\) hidden-layer biases, and \(c_i, d_i \in \mathbb{R}\) output-layer weights. The parameter vectors are defined as
\(
\eta = \{w_{i1}, w_{i2}, b_i, c_i\}_{i=1}^{m},~
\theta = \{a_{i1}, a_{i2}, \beta_i, d_i\}_{i=1}^{n}.
\)
The underlying energy functionals are optimized via a full gradient descent approach. This architectural choice reflects the analytical structure of the weak formulation: the smooth $\tanh$ activations ensure sufficient regularity to evaluate first-order spatial derivatives via automatic differentiation, enabling a stable approximation of the Dirichlet energy and the associated constraint residuals. Domain integrals appearing in the variational formulation are approximated by Monte Carlo quadrature over randomly sampled collocation points, yielding a mesh-free discretization of the underlying Sobolev inner products.

We choose the exact solution $u_{\text{exact}}(x,y) =\sin(\pi x)\sin(\pi y)$, which yields the source term $f(x,y)=2\pi^2\sin(\pi x)\sin(\pi y)$. The remaining hyperparameter configurations are summarized as: Uzawa outer iterations: $N = 80$, with inner optimization iterations $N_u = N_r = 300$, Uzawa step size: $\tau = 0.5$, Optimization learning rates: $\alpha = 0.005$ (for $u$) and $\omega = 0.005$ (for $r$).

We observe in Figure \ref{fig:u_approx_2d} that the  Uzawa approximated solution moves towards the exact solution over successive Uzawa iterations. 
\begin{figure}[htbp]
    \centering
    \includegraphics[width=0.5\linewidth]{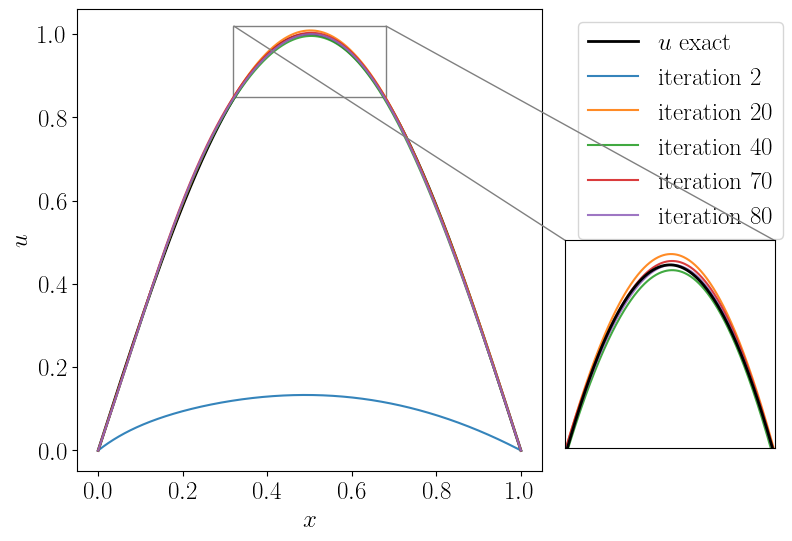}
    \caption{Evolution of the solution approximation $u_\theta^{k}(x)~\text{along}~y=0.5$ over successive Uzawa iterations.}
    \label{fig:u_approx_2d}
\end{figure}
The convergence performance of the Uzawa Double Deep Ritz method is illustrated in Figures \ref{fig:2D_uerror} and \ref{fig:res,sltn_2d}.

Figure \ref{fig:2D_uerror} monitors the square of the error norm, $\|u - u_{\text{exact}}\|^2$, against the cumulative inner optimization iterations (spanning a total of $N \times N_u = 24,000$ iterations). The staircase pattern reflects the interaction of the two loops: within each outer iteration, the inner gradient descent reduces the error until it plateaus, after which the outer Uzawa update shifts the target and enables a further decrease. During the first 5,000 iterations
the error decreases rapidly in this staircase fashion; beyond that point it stabilizes, exhibiting
oscillations typical of gradient-based optimization that can be reduced by hyperparameter
tuning.

\begin{figure}[htbp]
    \centering
    \begin{subfigure}[b]{0.4\textwidth}
    \centering
        \includegraphics[width=\linewidth]{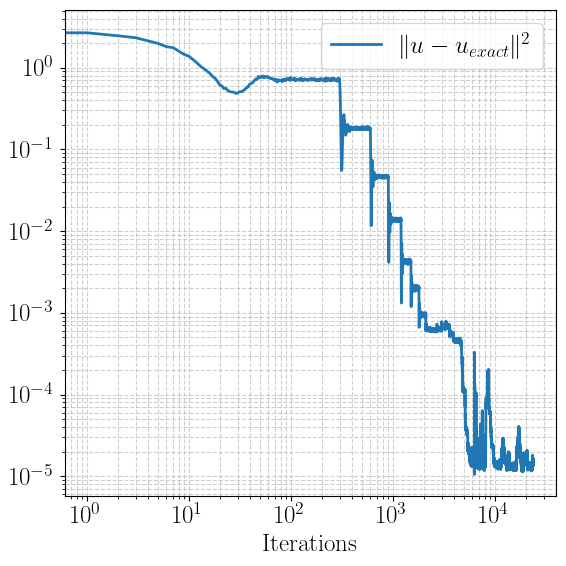}
        \subcaption{}
        \label{fig:2D_uerror}
    \end{subfigure}
    \begin{subfigure}[b]{0.4\textwidth}
        \includegraphics[width=\linewidth]{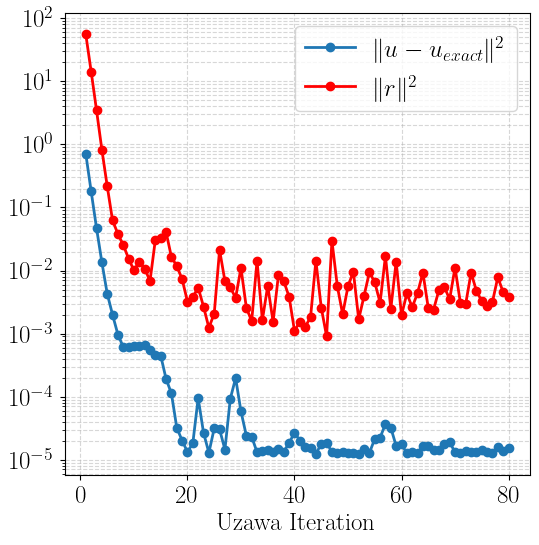}
        \subcaption{}
        \label{fig:res,sltn_2d}
    \end{subfigure}
    \caption{(a) Convergence of the solution error $\|u - u_{\text{exact}}\|^2$ for the 2D Poisson problem, plotted against the cumulative inner optimization iterations across the entire training process ($80 \text{ outer iterations} \times 300 \text{ inner steps}$). (b) Evolution of the residual norm and solution error over Uzawa iterations. }
\end{figure}

To better isolate the macroscopic behavior of the algorithm, Figure \ref{fig:res,sltn_2d} tracks the convergence with respect to the Uzawa outer iterations. We plot both the solution error $\|u -u_{\text{exact}}\|^2$ and the squared norm of the residual $\|r\|^2$. Both quantities decay simultaneously, confirming that residual minimization drives convergence of the solution approximation. Within the first 20 Uzawa iterations, the solution error drops by nearly five orders of magnitude, stabilizing at an error floor near $10^{-5}$. The simultaneous, rapid drop of both the solution error and the dual residual shows how quickly and reliably the proposed method converges. 

\section{Conclusion} 
\label{sec:conl}
We introduced the Uzawa Double Deep Ritz Method, a deep PDE solver that unifies Deep Ritz-type energy minimization with the classical Uzawa framework for saddle-point problems. By employing two neural networks to approximate the trial and test variables, the proposed method provides a mesh-free alternative to classical finite element residual-minimization techniques while retaining the structural stability associated with Uzawa iterations.

From a theoretical perspective, the contributions are threefold. First, we show that the proposed method can be interpreted as an iterative realization of the Weak Adversarial Network formulation, thereby connecting adversarial residual minimization with classical Uzawa iterations. Second, we establish convergence of an inexact Uzawa scheme where both residual and solution updates are computed only approximately. Third, we derive convergence guarantees for practical gradient-based training regimes, including single-step and multi-step updates; the single-step case is closely related to the Arrow–Hurwicz method, while the multi-step case reflects more accurate inner optimization.

The numerical experiments corroborate the theoretical findings and offer a detailed picture of the algorithmic behavior. The results demonstrate consistent decay of both the residual norm and the solution error, as well as stable convergence across different choices of step size and inner iteration parameters. In particular, the experiments highlight the robustness of the block-gradient strategy and illustrate the role of inner optimization accuracy in controlling convergence rates, thereby validating the practical relevance of the inexact analysis.

By bridging classical saddle-point theory with modern neural approximation, the proposed framework opens a pathway toward the systematic design of robust, consistent deep learning methods for partial differential equations.

\appendix
\section{Proof of \cref{thm3.1}}\label{sec:appendix}
\begin{proof}
Upon using the second update rule as in \cref{uzawa_operator_update} 
\begin{align*}
\|u^{k+1}-u^*\|^2_U &= \|u^{k}+\tau~R_U ^{-1}B^* r^{k}-u^{*}-\tau~R_U ^{-1}B^* r^{*}\|^2_U \\ 
& = \|u^{k}-u^{*}+\tau~R_U ^{-1}B^*(r^{k}-r^*)\|^2_U   
\end{align*}
Now, by definition of norm,
\begin{equation}\label{eq5}
    \|u^{k+1}-u^*\|^2_U=\|u^{k}-u^*\|^2_U+\tau~^{2}\|R_U ^{-1}B^*(r^{k}-r^*)\|^2_U+2\tau~(R_U ^{-1}B^*(r^{k}-r^*),u^k-u^*)_U
\end{equation}

\begin{claim}\label{claim_appendix}
Notice that, $(R_U ^{-1}B^*(r^{k}-r^*),u^k-u^*)_U=b(u^k-u^*,r^k-r^*)=-\|r^k-r^*\|^2_V$.
\end{claim}

\noindent Using claim~\ref{claim_appendix} in Eq.\cref{eq5}, we get
\[\|u^{k+1}-u^*\|^2_U=\|u^{k}-u^*\|^2_U+\tau~^{2}\|R_U ^{-1}B^*(r^{k}-r^*)\|^2_U+2\tau~(R_U ^{-1}B^*(r^{k}-r^*),u^k-u^*)_U\]
\[=\|u^{k}-u^*\|^2_U+\tau~^{2}\|R_U ^{-1}B^*(r^{k}-r^*)\|^2_U-2\tau~\|r^k-r^*\|^2_V\]
Using the boundedness property of $R_U^{-1}$ and $B^*$ i.e., 

$\|R_U^{-1}(g)\|_U=\|g\|_{U^\ast}$ and $\|B^* v\|_{U^\ast} \leq M \|v\|_{V} $
\[\|u^{k+1}-u^*\|^2_U \leq \|u^{k}-u^*\|^2_U+ \tau~^{2}M^2\|r^k-r^*\|^2_V-2\tau~\|r^k-r^*\|^2_V\]
\[=\|u^{k}-u^*\|^2_U+( \tau~^{2}M^2-2\tau~)\|r^k-r^*\|^2_V\]
\noindent Since for $\tau~\in (0,\frac{2}{M^2})$ we obtain $( \tau~^{2}M^2-2\tau~)\|r^k-r^*\|^2_V<0$ 
\[ \implies\|u^{k+1}-u^*\|^2_U <\|u^{k}-u^*\|^2_U\]
\[ \implies \lim_{k\rightarrow \infty} \|u^{k}-u^*\|^2_U- \|u^{k+1}-u^*\|^2_U=0 \]
Now,
\[\|r^k-r^*\|^2_V \leq \frac{\|u^{k}-u^*\|^2_U- \|u^{k+1}-u^*\|^2_U}{2\tau-\tau~^{2}M^2~}\]
\[\implies \lim_{k\rightarrow \infty} \|r^k-r^*\|^2_V=0. \]
\end{proof}

\begin{proof}[Proof of Claim~\ref{claim_appendix}]
\[(R_U ^{-1}B^*(r^{k}-r^*),u^k-u^*)_U=B^*(r^{k}-r^*)(u^k-u^*)=b(u^k-u^*,r^k-r^*)\]
Also, we have from the first update rule  for $r^*~and~r^k$:
\[(r^{k},v)_V=\ell(v)-b(u^{k},v)\]
\[(r^{*},v)_V=\ell(v)-b(u^{*},v)\]
On subtracting and substituting $v=r^k-r^*$ 
\[(r^k-r^*,r^k-r^*)_V=-b(u^k-u^*,r^k-r^*)\]
\[\|r^k-r^*\|^2_V=-b(u^k-u^*,r^k-r^*)\]
\[\implies b(u^k-u^*,r^k-r^*)=(R_U ^{-1}B^*(r^{k}-r^*),u^k-u^*)_U=-\|r^k-r^*\|^2_V.\]   
\end{proof}

\section*{Acknowledgments}
EBC is supported by the European Union’s Horizon Europe research and innovation programme under the Marie Sk\l{}odowska-Curie grant agreement No 101119556. The research of IB and KvdZ was supported by the Engineering and Physical Sciences Research Council (EPSRC), UK, under Grant EP/W010011/1.

\bibliographystyle{siamplain}
\bibliography{references}
\end{document}